\documentclass{amsart}
\usepackage{amssymb,amsfonts,latexsym}

\newtheorem{theorem}{Theorem}[section]
\newtheorem{lemma}[theorem]{Lemma}
\newtheorem{proposition}[theorem]{Proposition}
\newtheorem{corollary}[theorem]{Corollary}

\theoremstyle{definition}
\newtheorem{definition}[theorem]{Definition}

\newtheorem{observation}[theorem]{Observation}

\newtheorem{example}[theorem]{Example}

\newtheorem{conjecture}[theorem]{Conjecture}
\newtheorem{remark}[theorem]{Remark}
\newtheorem{general remarks}[theorem]{General remarks}

\newcommand{\Id}{\operatorname{Id}}

\renewcommand{\span}{\operatorname{span}}

\newcommand{\ben}{\begin{enumerate}}
\newcommand{\een}{\end{enumerate}}

\begin{document}

\title[On group gradings on PI-algebras]
{On group gradings on PI-algebras}

\author{Eli Aljadeff}
\thanks{The first author was supported by THE ISRAEL SCIENCE FOUNDATION (grant No. 1017/12).}

\address[Eli Aljadeff]{Department of Mathematics, Technion-Israel Institute of
Technology, Haifa 32000, Israel}
\email{aljadeff@tx.technion.ac.il}

\author{Ofir David}
\address[Ofir David]{Department of Mathematics, Technion-Israel Institute of
Technology, Haifa 32000, Israel}
\email{ofirdav@tx.technion.ac.il}

\subjclass[2000]{}


\begin{abstract}

We show that there exists a constant $K$ such that for any
PI-algebra $W$ and any nondegenerate $G$-grading on $W$ where $G$
is any group (possibly infinite), there exists an abelian subgroup
$U$ of $G$ with \\ $[G:U] \leq \exp(W)^{K}$. A $G$-grading
$W=\oplus_{g \in G} W_{g}$ is said to be nondegenerate if
$W_{g_1}W_{g_2}\cdots W_{g_r} \neq 0$ for any $r \geq
1$ and any $r$ tuple $(g_1, g_2, \ldots, g_r)$ in
$G^{r}$.

\end{abstract}

\maketitle

\bigskip
\noindent

\bigskip\bigskip

\hspace{3cm}

\begin{section}{Introduction} \label{Introduction}

In the last two decades there were significant efforts to extend
important results in the theory of polynomial identities for
(ordinary) associative algebras to $G$-graded algebras, where $G$
is a \textit{finite group}, and more generally to $H$-comodule
algebras where $H$ is a finite dimensional Hopf algebra. For
instance Kemer's representability theorem and the solution of the
Specht problem were established for $G$-graded associative
algebras over a field of characteristic zero (see \cite{AB},
\cite{Svi}). Recall that Kemer's representability theorem says
that any associative PI-algebra over a field $F$ of characteristic
zero is PI-equivalent to the Grassmann envelope of a finite
dimensional $\mathbb{Z}_{2}$-graded algebra $A$ over some field
extension $L$ of $F$ (see below the precise statement and
Proposition \ref{reduction from arbitrary to finite dimensional
algebras}). Another instance of these efforts is the proof of
Amitsur's conjecture which was originally proved for ungraded
associative algebras over $F$ by Giambruno and Zaicev \cite{GZ},
and was extended to the context of $G$-graded algebras by
Giambruno, La Mattina and the first named author of this article
(see \cite{AGL}, \cite{GiaLaMattina}) and considerable more
generally for $H$-comodule algebras by Gordienko \cite{Gor}.
Amitsur's conjecture states that the sequence $c_{n}^{1/n}$, where
$c_{n}=c_{n}(W)$ is the $n$th term of the codimension sequence of
$W$, has an integer limit (denoted by $\exp(W)$).

In \cite{Al-David} a different point of view was considered (in
combining PI-theory and $G$-gradings, still under the condition
that $G$ is finite), namely asymptotic PI-theory was applied in
order to prove invariance of the order of the grading group on an
associative algebra whenever the grading is minimal regular (as
conjectured by Bahturin and Regev \cite{BaRe}). In fact, it is
shown there that the order of the grading group coincides with
$\exp(W)$.

Suppose now $G$ is arbitrary (i.e. not necessarily finite). Our
goal in this paper, roughly speaking, is to exploit the invariant
$\exp(W)$ of the algebra $W$, in order to put an effective bound
on the minimal index of an abelian subgroup of $G$ whenever the
algebra $W$ admits a $G$-gradings satisfying a natural condition
which we call ``nondegenerate'' (see Definition \ref{nondegenerate
definition}). Our results extend considerable known results for PI
group algebras (which are obviously nondegenerately $G$-graded).
Let us remark here that a big part of our analysis is devoted to
the case where the group $G$ is finite (a case where Kemer and
asymptotic PI theory can be applied) and then we pass to infinite
groups.

In this paper we only consider fields of characteristic
zero. Let $W$ be an associative PI-algebra over a field $F$.
Suppose $W\cong \bigoplus_{g \in G}W_{g}$ is $G$-graded where $G$
is arbitrary.

\begin{definition} \label{nondegenerate definition}

We say the $G$-grading on $W$ is \textit{nondegenerate} if for any positive integer $r$ and
any tuple $(g_{1},\ldots, g_{r}) \in G^{(r)}$, we have $W_{g_{1}}W_{g_{2}}\cdots W_{g_{r}}\neq 0$.
\end{definition}

\begin{theorem} \label{thm:Main_Theorem}$(Main$ $theorem)$

There exists an integer $K$ such that for any PI-algebra $W$ and
for any "nondegenerate'' $G$-grading on $W$ by any group $G$,
there exists an abelian subgroup $U$ of $G$ with $[G:U] \leq \exp(W)^{K}$.

\end{theorem}

It is known (and not difficult to prove; see Lemma \ref{characteristic subgroup}) that if a
group $G$ has an abelian subgroup of index $n$, then it contains a
characteristic abelian subgroup whose index is bounded by a
function of $n$. We therefore have the following corollary.

\begin{corollary}

There exists a function $f:\mathbb{N}\rightarrow \mathbb{N}$ such that for any PI-algebra $W$ and for any "nondegenerate'' $G$-grading on $W$ by any group $G$,
there exists a characteristic abelian subgroup $U$ of $G$ with $[G:U] \leq f(\exp(W))$.

\end{corollary}

In order to put our main result in an ``appropriate'' context, we
recall $(i)$ different type of $G$-gradings on associative
algebras $(ii)$ three conditions on groups which are closely
related to the content of the main theorem, namely
$n$-permutability, $n$-rewritability and $PI_{n}$ (the group
algebra $FG$ satisfies a polynomial identity of degree $n$, $char(F)=0$).

A $G$-grading on $W\neq 0$ is called \textit{strong }if
$W_{g}W_{h}=W_{gh}$ for every $g,h \in G$. Note that this
condition is considerable stronger than a nondegenerate grading.
For instance, the well known $\mathbb{Z}_{2}$-grading on the
infinite dimensional Grassmann algebra is nondegenerate but not
strong. The fact that the $\mathbb{Z}_{2}$-grading on the
Grassmann algebra is nondegenerate will play an important role in
the proof of the main theorem. Strong grading is considerably
weaker than \textit{crossed product grading} which requires that
every homogeneous component has an invertible element (e.g. group
algebras). In the other direction we may consider conditions on
$G$-gradings which are weaker than nondegenerate $G$-gradings as
$G$-gradings where $W_{g}\neq 0$ for every $g \in G$ (call it
\textit{connected grading}). A somewhat stronger condition to the
latter but yet weaker than nondegenerate grading is a condition
which we call bounded nondegenerate: by definition a $G$-grading
on an algebra $W$ is \textit{bounded nondegenerate} if any product
of homogeneous components $W_{g_{1}}W_{g_{2}}\cdots W_{g_{r}}$
does not vanish unless $r > r_{0}$ for some (large) fixed integer
$r_{0}$. We thus have \textit{\textit{crossed product grading}
$\Rightarrow$ \textit{strong grading} $\Rightarrow$ nondegenerate}
\textit{grading} $\Rightarrow$ \textit{bounded nondegenerate
grading} $\Rightarrow$ \textit{connected grading}.

 In section \ref{sec:examples} we show that if a PI algebra $W$ is
``bounded nondegeneratly'' $G$-graded than the main theorem is
false in general.

\begin{definition} (see \cite{CLMR, El-Pass, Pass})
Let $n>1$ be an integer.
\begin{enumerate}

\item

We say that a group $G$ is $n$-permutable (resp. $n$-rewritable),
denoted by $P_{n}$ (resp. $Q_{n})$, if for any $n$-tuple
$(g_{1},\ldots, g_{n}) \in G^{(n)}$ there exists a nontrivial
permutation $\sigma \in Sym(n)$ (resp. distinct
permutations $\sigma, \tau \in Sym(n)$) such that
$$
g_{1}g_{2}\cdots g_{n}= g_{\sigma(1)}g_{\sigma(2)}\cdots
g_{\sigma(n)} \in G
$$

(resp.
$$
g_{\sigma(1)}g_{\sigma(2)}\cdots
g_{\sigma(n)}= g_{\tau(1)}g_{\tau(2)}\cdots
g_{\tau(n)} \in G).
$$

\item
We say that a group $G$ satisfies $PI_{n}$ if the group algebra
$FG$ satisfies a (multilinear) identity of degree $n$ (it is well
known that since $F$ is a field of characteristic zero, the
$T$-ideal of identities is generated by multilinear polynomials).

\end{enumerate}

\end{definition}

Clearly, $P_{n} \Rightarrow P_{n+1}$, $Q_{n} \Rightarrow Q_{n+1}$
and $P_{n} \Rightarrow Q_{n}$. We say that group is permutable
(resp. rewritable, PI), if it is $n$-permutable (resp.
$n$-rewritable, $PI_{n}$) for some $n$. We denote
(with a slight abuse of notation) by $P$, $Q$, $PI$ the families
of all permutable, rewritable or PI groups. It was proved in
\cite{El-Pass} that if a group is $n$-rewritable then it is
$m$-permutable where $m$ is bounded by a function of $n$.

As for the condition $PI_{n}$, it is easy to show that if $FG$
satisfies a (multilinear) polynomial identity of degree $n$ then
the group $G$ is $n$-permutable and in particular $n$-rewritable
(indeed, if $f(x_1,\ldots,x_n)=x_1\cdots x_n + \sum_{e\neq\sigma
\in Sym(n)}\alpha_{\sigma}x_\sigma(1)\cdots x_\sigma(n)$,
$\alpha_{\sigma}\in F$, is a multilinear identity of $FG$ and
$(g_1,\ldots,g_{n}) \in G^{(n)}$ is any $n$th tuple, the
evaluation $x_{i}=g_{i}$, $i=1,\ldots,n$, yields $g_1\cdots g_n =
g_{\sigma(1)}\cdots g_{\sigma(n)}$ for some $e \neq \sigma \in
Sym(n)$). Thus we have that $PI_{n} \Rightarrow P_{n} \Rightarrow
Q_{n}$. As for the reverse direction of arrows the following is
known (see \cite{El-Pass}).

\begin{enumerate}
\item
$Q_{n}$ is strictly weaker than $P_{n}$ (although, as mentioned above, there exists a function $f$ such that $Q_{n}
\Rightarrow P_{f(n)}$).
\item
$P_{n} \nRightarrow PI_{m}$
for any $n$ and $m$. In particular it is known that if $G$
satisfies $PI_{m}$, then $G$ has a finite index abelian subgroup
whose index is bounded by a function of $m$ whereas for any $n > 2$ there exists an infinite family of finite groups $\{G_{i}\}_i$
which satisfy $P_{n}$, whose PI degree is $d_{i}$ and $\lim
 d_{i}=\infty$ (the PI degree of $G$ is the minimal degree of a
nontrivial polynomial identity of $FG$).

 \end{enumerate}

\begin{remark} \label{f.g. or non f.g. rewrite-permute}
As for the existence of a finite index abelian subgroup in $G$ and
the permutability or rewritability conditions there is an
interesting distinction between finitely/nonfinitely generated
groups. If $G$ is finitely generated and satisfies $P_{n}$ (or
$Q_{n}$) then it has an abelian subgroup of finite index (note
however, as mentioned above, the index is not bounded by a
function of $n$ ; see example in section \ref{sec:examples}). If
$G$ is not finitely generated, it may not have a finite index
abelian subgroup. However it does have a characteristic subgroup
$H$ whose index $[G:H]$ is bounded by a function of $n$ and whose
commutator subgroup $H'$ is finite, and its order is bounded by a
function of $n$.

\end{remark}

In view of the above considerations it is natural to introduce the
following condition on a group $G$.

\begin{definition}
Let $G$ be any group. We say that $G$ satisfies $T_{n}$ if there
exists a PI algebra $W$ of PI degree $n$ which admits a
nondegenerate $G$-grading. We say that $G$ has
$T$ if it has $T_{n}$ for some $n$.

\end{definition}

It is easy to see that the argument which shows $PI_{n}
\Rightarrow P_{n}$ shows also that $T_{n} \Rightarrow P_{n}$. This
simple fact will play an important role while extending the proof
of the main theorem from finitely generated residually finite
groups to arbitrary finitely generated groups.

Note that since the group algebra $FG$ is nondegenerately
$G$-graded we have $PI_{n} \Rightarrow T_{n}$. In the other direction it follows from our main theorem that if $G$ satisfies $T_{n}$,
then $G$ has $PI_{m}$ for some $m$ (indeed, $G$ is abelian by finite and hence, by \cite{Kap}, the group algebra $FG$ is PI).
As for the relation between $m$ and $n$ we have the following result.

\begin{theorem}
Let $G$ be any group and suppose it grades nondegenerately a PI
algebra $W$ of PI degree $n$. Then the group algebra $FG$ is PI
and its PI degree is bounded by $n^{2}$. Similarly, $\exp(FG) \leq
\exp(W)^{2}$.

\end{theorem}

It is somewhat surprising that $T_{n} \nRightarrow PI_{n}$
(intuitively, the group algebra $FG$ seems to be the ``smallest or
simplest'' $G$-graded algebra whose grading is nondegenerate). The
following example shows that twisted group algebras may have lower
PI degree.

\begin{example} \label{binary thetraeder}

Let $A_{4}$ be the alternating group of order $12$. It is known
that the largest irreducible complex representation is of degree
$3$ and hence by Amitsur-Levitzky theorem the PI degree is $6$. On
the other hand, the group $A_{4}$ admits a nontrivial cohomology
class $\alpha \in H^{2}(A_{4}, \mathbb{C}^{*})$ (corresponding to the
binary tetrahedral group of order $24$). Since twisted group
algebras with nontrivial cohomology class cannot admit the trivial
representation, we have $\mathbb{C}^{\alpha}A_{4}\cong M_{2}(\mathbb{C}) \oplus
M_{2}(\mathbb{C}) \oplus M_{2}(\mathbb{C})$ and hence the PI degree is $4$.

\end{example}

\begin{conjecture}

Let $W$ be a an algebra over an algebraically closed field $F$ of
characteristic zero satisfying a PI of degree $n$. Suppose $W$ is
nondegenerately graded by a group $G$. Then there exists a class
$\alpha \in H^{2}(G,F^{*})$ such that the twisted group algebra
$F^{\alpha}G$ has PI degree bounded by the same integer $n$.
\end{conjecture}

\begin{theorem} Notation as above.
The conjecture holds whenever the group $G$ is finite.

\end{theorem}

The main tools used in the proof of the main theorem are the
representability theorem for $G$-graded algebras where $G$ is a
\textit{finite group} \cite{AB} and Giambruno and Zaicev's result
on the exponent of $W$ \cite{GZ}. The representability theorem
allows us to replace the $G$-graded algebra $W$ by a finite
dimensional $G$-graded algebra $A$ (or the Grassmann envelope of a
finite dimensional $\mathbb{Z}_{2}\times G$-graded algebra $A$)
whereas Giambruno and Zaicev's result provides an interpretation
of $\exp(W)$ in terms of the dimension of a certain subalgebra of
$A$. The proof of Theorem \ref{thm:Main_Theorem} in case the group
$G$ is finite is presented in section $3$. In section $4$ we show
how to pass from finite groups to arbitrary groups and by this we
complete the proof of Theorem \ref{thm:Main_Theorem}.

In section $2$ we recall some background on group gradings and PI
theory needed for the proofs of the main results of the paper. In the last section of the paper, section $\ref{sec:examples}$, we present
(1) a family of $n$-permutable with no uniform bound on the index
of abelian subgroups and (2) an example which shows that we cannot
replace in the main theorem nondegenerate $G$-gradings with
bounded nondegenerate $G$-grading.

We close the introduction by explaining why one would prefer
bounding the index of an abelian subgroup by a function of the
$\exp(W)$ (as in the main theorem) rather than by the PI degree of
$W$. It is known that $\exp(W)$ is bounded by a function of the PI
degree (e.g. $\exp(W) \leq (d(W)-1)^{2}$, see Theorem 4.2.4
\cite{GiamZai}) but such function does not exist in the reverse
direction. Indeed, since $\exp(W)$ is an \textit{asymptotic} invariant it
remains invariant if we consider the $G$-graded $T$-ideal
generated by all polynomials in $\Id_{G}(W)$ of degree at least
$m$ (any $m$) whereas the $G$-graded PI degree and hence the
ordinary PI degree is at least $m$.

\end{section}

\begin{section}{Background and some preliminary reductions} \label{Background and some preliminary reductions}

We start by recalling some facts on $G$-graded algebras $W$ over a
field $F$ of characteristic zero and their corresponding
$G$-graded identities. We refer the reader to \cite{AB} for a
detailed account on this topic.

\begin{remark}

In this section we consider only finite groups. Although some of
the basic results in $G$-graded PI theory hold for arbitrary
groups, one of our main tools, namely the ``representability
theorem'' for $G$-graded PI algebras, is false for infinite
groups.

\end{remark}

\begin{subsection}{$G$-graded identities}

Let $W$ be a PI-algebra over $F$. Suppose $W$ is $G$-graded where
$G$ is a finite group. Denote by $I=\Id_{G}(W)$ the ideal of
$G$-graded polynomial identities of $W$. It consists of all elements in the free
$G$-graded algebra $F\langle X_{G} \rangle$ over $F$, that vanish
upon any admissible evaluation on $W$. Here, $X_{G}=\bigcup_{g \in
G} X_{g}$ and $X_{g}$ is a set of countably many variables of
degree $g$. An evaluation on $W$ is admissible if the variables
from $X_{g}$ are replaced only by elements of $W_{g}$. The ideal
$I$ is a $G$-graded $T$-ideal, i.e. it is invariant under all $G$-graded
endomorphisms of $F \langle X_{G} \rangle$.

We recall from \cite{AB} that the $G$-graded $T$-ideal $I$ is
generated by multilinear polynomials. Consequently, it remains
invariant when passing to any field extension $L$ of $F$, that is $\Id_{G}(W\otimes_{F}L)=\Id_{G}(W)\otimes_{F}L$.

The following observations play an important role in the proofs.

\begin{observation}\label{obs:equivalent_nondeg}
The condition \textit{nondegenerate} $G$-\textit{grading} on $W$ can be easily
translated into the language of $G$-graded polynomial identities.
Indeed a $G$-grading on $W$ is nondegenerate if and only if for
any integer $r$ and any tuple $(g_{1},\ldots, g_{r}) \in G^{(r)}$,
the $G$-graded multilinear monomial $x_{g_{1},1}\cdots x_{g_{r},r}$ is
a $G$-graded \textit{nonidentity} of $W$ (in short we say that $\Id_{G}(W)$
contains no multilinear $G$-graded monomials). Consequently, if
$G$-graded algebras $W_{1}$ and $W_{2}$ are $G$-graded
PI-equivalent (i.e. have the same $T$-ideal of $G$-graded
identities), then the grading on $W_{1}$ is nondegenerate if and
only if the grading on $W_{2}$ is nondegenerate.

\end{observation}

\begin{observation}\label{obs:equivalent_exp}
If $W_1, W_2$ are two $G$-graded algebras with
$\Id_G(W_1)=\Id_G(W_2)$, then $\Id(W_1)=\Id(W_2)$ (the ungraded
identities). In particular we have $\exp(W_1)=\exp(W_2)$. Indeed,
this follows easily from the fact that a polynomial
$p(x_{1},\ldots,x_{n})$ is an ungraded identity of an algebra $W$
with a $G$-grading if and only if the polynomial $p(\sum_{g\in
G}x_{g,1}\ldots,\sum_{g\in G}x_{g,n})$ is a graded identity of $W$
as a $G$-graded algebra.

\end{observation}

As noted above, the nondegeneracy condition satisfied by a
$G$-grading on $W$ depends only on the $T$-ideal of $G$-graded
identities, hence if the grading on a $G$-graded algebra $W$ over
a field $F$ is nondegenerate, the same holds for the $G$-graded
algebra $W_{L}=W \otimes_{F}L$. Similarly, the numerical invariant
$\exp(W)$ of the algebra $W$ remains unchanged if we extend
scalars.

\begin{remark}
In the main steps of the proof (in case $G$ is a finite group),
roughly speaking, we ''pass'' to simpler algebras
without increasing too much the exponent the PI
degree. More precisely, given an arbitrary $G$-grading on a PI
algebra $W$ we first pass to a finite dimensional $G$-graded
algebra $A$, then to a $G$-simple algebra and finally
to a group algebra, a case which was solved by Gluck using the classification of finite simple groups (see
\cite{Glu}).

\end{remark}

Let us recall some terminology and some facts from Kemer's
theory extended to the context of $G$-graded algebras as they
appear in \cite{AB}.

Let $W$ be a $G$-graded algebra over $F$. Suppose that $W$ is PI
(as an ungraded algebra). Kemer's representability theorem for
$G$-graded algebras assures that there exists a field extension
$L/F$ and a finite dimensional $\mathbb{Z}_{2} \times G$-graded
algebra $A$ over $L$ such that the Grassmann envelope $E(A)$ (with
respect to the $\mathbb{Z}_{2}$-grading) yields a $G$-graded
algebra which is $G$-graded PI-equivalent to $W_L$ (see Proposition \ref{reduction from arbitrary to finite dimensional algebras}). In case the
algebra $W$ is affine, or more generally in case it satisfies a
Capelli identity (it is known that any affine PI algebra satisfies a Capelli identity), there exists a field extension $L/F$ such that
the algebra $W_L$ is $G$-graded PI-equivalent to a finite
dimensional $G$-graded algebra $A$ over $L$. This result will be
used to reduce our discussion from infinite dimensional algebras
to finite dimensional ones in case the group $G$ is finite. As
extensions of scalars do not change the exponent (nor the
PI-degree) we assume that the field $L$ is algebraically closed.
\end{subsection}

\begin{subsection} {$G$-simple algebras}
The next ingredient we need is a result of Bahturin, Sehgal and
Zaicev, which determines the $G$-graded structure of finite
dimensional $G$-simple algebra over an algebraically closed field
of characteristic zero.

Let $A$ be the algebra of $r \times r$-matrices over $F$
and let $G$ be any group (here, $G$ may be infinite). Fix an $r$-tuple $\alpha=(g_{1},\ldots,g_{r}) \in G^{(r)}$. Consider the
$G$-grading on $A$ given by

$$
A_{g}=\span_{F}\{e_{i,j}: g=g_{i}^{-1}g_{j}\}.
$$
One checks easily that this indeed determines a $G$-grading on
$A$. Clearly, since the algebra $A$ is simple, it is $G$-simple as a $G$-graded algebra.

Next we present a different type of $G$-gradings on semisimple
algebras which turn out to be $G$-simple. Let $H$ be any finite
subgroup of $G$ and consider the group algebra $FH$. By Maschke's
theorem $FH$ is semisimple and of course $H$-simple (any nonzero
homogeneous element is invertible). More generally we consider
twisted group algebras
$F^{\alpha}H$ as $H$-graded algebras, where $\alpha$ is a
$2$-cocycle in $Z^{2}(H, F^{*})$ ($H$ acts trivially on $F$).
Recall that $F^{\alpha}H = \span_{F} \{U_{h}: h \in H\}$,
$U_{h_{1}}U_{h_{2}}=\alpha(h_{1},h_{2})U_{h_{1}h_{2}}$, for all
$h_{1},h_{2} \in H$. We say that the basis $\{U_{h}:h \in H\}$
\textit{corresponds to the $2$-cocycle} $\alpha$. Finally, we may view the twisted group algebra $F^{\alpha}H$ as a $G$-graded algebra by setting $A_{g} = 0$ for $g
\in G \setminus H$ and as such it is $G$-simple. We
refer to the $G$-grading on $F^{\alpha}H$ as a \textit{fine grading} (i.e. every homogeneous component is of dimension $\leq 1$).

\begin{remark}
In the sequel, whenever we say that $\{U_{h}: h\in H\}$ is a basis
of $F^{\alpha}H$, we mean that the basis corresponds to the
cocycle $\alpha$. One knows that in general an homogeneous basis
of that kind corresponds to a cocycle $\alpha{'}$ cohomologous to
$\alpha$.

\end{remark}

In case the field $F$ is algebraically closed
of characteristic zero, we have that these two gradings
(\textit{elementary} and \textit{fine}) are the building blocks of
any $G$-grading on a finite dimensional algebra so that it is
$G$-simple. This is a theorem of Bahturin, Sehgal and Zaicev.

\begin{theorem} \cite{BSZ} \label{thm:simple_graded}
Let $A$ be a finite dimensional $G$-graded simple algebra. Then
there exists a finite subgroup $H$ of $G$, a $2$-cocycle $\alpha
:H\times H\rightarrow F^{*}$ where the action of $H$ on $F$ is
trivial, an integer $r$ and an $r$-tuple
$(g_{1},g_{2},\ldots,g_{r})\in G^{(r)}$ such that $A$ is
$G$-graded isomorphic to $\Lambda=F^{\alpha}H\otimes M_{r}(F)$
where $\Lambda_{g}=span_{F}\{U_{h}\otimes e_{i,j} \mid
g=g_{i}^{-1}hg_{j}\}$. Here $U_{h}\in F^{\alpha}H$ is a
representative of $h\in H$ and $e_{i,j}\in M_{r}(F)$ is the
$(i,j)$ elementary matrix.

In particular the idempotents $1\otimes e_{i,i}$ as well as the
identity element of $A$ are homogeneous of degree $e\in G$.
\end{theorem}

\end{subsection}

\begin{subsection} {Asymptotic PI-theory}
The last ingredient we need is Regev, Giambruno and Zaicev's
PI-asymptotic theory. Let $W$ be an ordinary PI-algebra over an
algebraically closed field $F$ of characteristic zero and let
$\Id(W)$ be its $T$-ideal of identities. Consider the
$n!$-dimensional vector space

$$
P_{n}= span_{F}\{x_{\sigma(1)}\cdots x_{\sigma(n)}: \sigma \in
Sym(n) \}$$ and let $c_{n}(W) = \dim_{F}(P_{n}/P_{n} \cap \Id(W))$
be the $n$-th term of the codimension sequence of the algebra $W$.
It was proved by Regev in 72 \cite{Reg} that the sequence
$\{c_{n}(W)\}$ is exponentially bounded and conjectured by Amitsur
that the limit $lim_{n\rightarrow \infty} c_{n}^{1/n}$ exists (the
exponent of $W$) and is a nonnegative integer. The conjecture was
established by Giambruno and Zaicev in the late 90's by showing
that the limit coincides, roughly speaking, with the dimension of
a certain subspace ``attached'' to $W$.  In particular, for a matrix
algebra $M_d(F)$ we have $exp(M_d(F))=d^2$ and by the Amitsur
Levitzky theorem it has PI-degree $2d$. Any finite dimensional
$G$-simple algebra is a direct product of matrix algebra (as an
ungraded algebra), hence its $T$-ideal of identities coincides with the ideal of identities (and therefore
the exponent and PI-degree) of the
largest matrix algebra appearing in its decomposition.

\begin{remark}
It follows from the Amitsur-Levitzki theorem that if $A$ is a finite dimensional $G$-simple algebra $A$
we have $\exp(A)=\frac{1}{4}(PIdeg(A))^2$. For an arbitrary
PI-algebra we only have the bound $\exp(A)\leq (PIdeg(A)-1)^2$. Recall (from the last paragraph of the introduction) that
the PI-degree cannot be bounded from above by any function
of $\exp(A)$.
\end{remark}

\end{subsection}
\end{section}

\begin{section}{Proof of main theorem-Finite groups} \label{Proofs}

All groups considered in this section are finite.

For a PI-algebra $W$ over a field $F$ of characteristic zero
we denote by $\exp(W)$ its exponent.

\begin{proposition} \label{reduction from arbitrary to finite dimensional algebras}

Let $W$ be a PI-algebra over a field $F$. Suppose $W$ is graded
nondegenerately by a group $G$. Then there exists a field
extension $L$ of $F$ and a finite dimensional $L$-algebra $W_0$ which
is nondegenerately $G$-graded, such that $exp(W_0)\leq exp(W)$.
\end{proposition}

\begin{proof}

Let us consider first the case where $W$ is affine. Applying
\cite{AB} there exists a finite dimensional $G$-graded
algebra $B$ over a field extension $L$ of $F$ such that $\Id_{G}(W\otimes_{F} L)=\Id_{G}(B)$.
Clearly, we may assume that $L$ is algebraically closed by further extending the scalars if needed.
Next, by Observations \ref{obs:equivalent_nondeg} and \ref{obs:equivalent_exp}
we know that the $G$-grading on $B$ is nondegenerate and $\exp(W_{L})=\exp(B)$,
thus proving the proposition for this case.

Suppose now that $W$ is arbitrary (i.e., not necessarily affine). By
\cite{AB} there exists a finite dimensional $\mathbb{Z}_{2} \times
G$-graded algebra $C\cong \bigoplus_{(\epsilon, g)\in \mathbb{Z}_{2} \times
G}C_{(\epsilon, g)}$ over an (algebraically closed) field extension $L$ of $F$ such that $W_L$ is $G$-PI-equivalent to
$E(C)=(E_{0} \otimes C_{0}) \oplus (E_{1} \otimes C_{1})$ (the Grassmann envelope of $C$) where $C_0=\bigoplus_{g \in
G}C_{(0,g)}$ and $C_1=\bigoplus_{g \in
G}C_{(1,g)}$. The
$G$-grading on $E(C)$ is given by
$$E(C)_{g}=(E_{0}\otimes C_{(0,g)})\oplus (E_{1}\otimes C_{(1,g)}).$$
We claim that the $G$-grading on $C$ is nondegenerate (where $C_g=C_{(0,g)}\oplus C_{(1,g)}$).
To this end fix an $n$-th tuple $(g_{1}, \ldots, g_{n}) \in G^{(n)}$.
By linearity we need to show that at least one of the $2^{n}$ monomials of the form

$$x_{(\epsilon_{1}, g_{1}),1}x_{(\epsilon_{2}, g_{2}),2}\cdots x_{(\epsilon_{n}, g_{n}),n}$$
is not in $\Id_{\mathbb{Z}_{2} \times G}(C)$. Let us show that if this is
not the case, then the monomial $x_{g_{1},1}\cdots x_{g_{n},n}$ is a
$G$-graded identity of $E(C)$,
contradicting the fact that the $G$-grading on $E(C)$ and hence on
$W$ is nondegenerate. To see this consider the evaluation
$x_{g_i,i}=z_{0,i}\otimes a_{0,i} + z_{ 1,i}\otimes a_{1,i}$ for
$i=1,\ldots,n$ where $z_{\epsilon,i}\in E_{\epsilon}$ and $a_{\epsilon,i} \in C_{(\epsilon,g_i)}$. This evaluation yields an expression with $2^{n}$
summands of the form
$$z_{\epsilon_1,i_1}z_{\epsilon_2,i_2}\cdots z_{\epsilon_n,i_n} \otimes a_{(\epsilon_1,i_1)}a_{(\epsilon_2,i_2)}\cdots a_{(\epsilon_n,i_n)}
$$
which are all zero and the claim follows.

Finally, by a theorem of Giambruno and Zaicev (see \cite{GZ}, proof of main theorem  or \cite{Alja}, Theorem 2.3) we have that
$\exp(C)\leq \exp_{\mathbb{Z}_2}(C)=\exp(E(C))=\exp(W)$, which is
precisely what we need.

\end{proof}

Our next step is to reduce the main theorem from finite dimensional algebras to $G$-simple algebras.

\begin{proposition} \label{reduction from finite dimensional to G-simple}

Let $W$ be a finite dimensional PI $F$-algebra graded nondegenerately by a group $G$.
Then there exists a $G$-simple algebra $W_0$ such that $Id_G(W)\subseteq Id_G(W_0)$ and
the grading on $W_0$ is nondegenerate $($in fact $W_{0}$ is an homomorphic image of $W$$)$. In particular $\exp(W_0)\leq \exp(W)$.

\end{proposition}

\begin{proof}
Denote by $J=J(W)$ the Jacobson radical of $W$. Since the
characteristic of the field is zero, it is known that $J$ is
$G$-graded and so $W/J$ is a semi-simple $G$-graded algebra (see
\cite{CoMon}).

We claim that the $G$-grading on $W/J$ is still nondegenerate.
Indeed, if $W/J$ satisfies a monomial identity $f$, then any
evaluation of this monomial on $W$ yields an element in $J$. Since
$J$ is nilpotent (say of nilpotency degree is $k$) we have that
the product of $k$ copies of $f$ (with distinct variables) is a
monomial identity of $W$. This contradicts the assumption the
$G$-grading on $W$ is nondegenerate and the claim is proved.

The algebra $W/J$ is $G$-semisimple and therefore a direct product
of $G$-simple algebra $\prod_1^n A_i$. If for each $i$ there is a
multilinear monomial identity $f_i$ of $A_i$, then the product
$\prod f_i$ is a multilinear monomial identity of $W/J$,
contradicting our assumption on the $G$-grading on $W/J$.
Consequently, there is an $i$ such that $A_{i}$ is nondegenerately
$G$-graded. Letting $W_{0}=A_{i}$ we have $Id_G(W)\subseteq
Id_G(W/J)\subseteq Id_G(W_0)$ as desired.

\end{proof}

In the next lemma we characterize (in terms of Bahturin,
Sehgal and Zaicev's theorem) when the grading on a $G$-simple algebra is nondegenerate.
Recall
that a $G$-grading on $A$ is \textit{strong} if for any $g, h \in G$ we
have $A_{g}A_{h}=A_{gh}$.

\begin{lemma} \label{lem:nondeg_simple}
Let $A\neq 0$ be a finite dimensional $G$-simple algebra. Then the following conditions are equivalent.

\begin{enumerate}

\item

The $G$-grading on $A$ is nondegenerate.

\item

The $G$-grading on $A$ is strong. In particular $A_{g} \neq 0$, for every $g\in G$.

\item

Let $F^{\alpha}H \otimes M_{r}(F)$ be a presentation of the
$G$-grading on $A$ $($as given by Theorem  \ref{thm:simple_graded}$)$ where $H$ is a finite
subgroup of $G$ and $(g_1,\ldots,g_{r})\in G^{(r)}$ is the $r$-tuple
which determines the elementary grading on $M_{r}(F)$. Then every
right coset of $H$ in $G$ is represented in the $r$-tuple.

\end{enumerate}

\end{lemma}

\begin{remark}
Note that in general (i.e. in case the algebra $A$ is not necessarily
$G$-simple) the first two conditions are not equivalent. For
instance (as mentioned in the introduction), the $\mathbb{Z}_{2}$-grading on the infinite dimensional
Grassmann algebra is nondegenerate but not strong. Indeed,
$E_{1}E_{1}\subsetneq E_{0}$ (or $E_{0}E_{0}\subsetneq E_{0}$ in
case the algebra $E$ is assumed to have no identity element).
\end{remark}
\begin{proof}
Note that since $A$ is assumed to be finite dimensional
$G$-simple, each one of the conditions (1)-(3) implies that $G$ is
finite. As for the $3$rd condition of the lemma we replace (as we
may by \cite{Al-Haile}, Lemma 1.3) the given presentation with
another so that the $r$-tuple has the following form

$$(g_{(1,1)},\dots,g_{(1,d_1)},
g_{(2,1)},\dots,g_{(2,d_2)},\ldots,\ldots g_{(s,1)},\dots,g_{(s,d_s)})$$
where
\begin{itemize}

\item

$r=d_{1}+\cdots + d_{s}$.

\item
$g_{i,1}=g_{i,2}=\ldots =g_{i,d_{i}}$ (denoted by $z_{i}$), and for $i \neq k$ the elements $g_{i,j}, g_{k,l}$
represent different right $H$-cosets in $G$.

\item $g_{1,j}=e$ for $j=1,\ldots,d_1$.\\

\end{itemize}

$(2)\rightarrow (1):$ This is clear.

$(1)\rightarrow(3):$ Suppose $(3)$ does not hold. We claim there
exists a multilinear monomial of degree at most $r$ which is a
$G$-graded identity of $F^{\alpha}H\otimes
M_{r}(F)=\span_{F}\{U_{h}\otimes e_{i,j}: h \in H, 1\leq i,j \leq
r\}$.

It is convenient to view the matrices in $M_{r}(F)$ as $s
\times s$ block matrices corresponding to the decomposition
$d_1+\cdots +d_s = r$. More precisely, let $D_k = d_1 + \cdots
+d_k$ and decompose $M_r(F)=\bigoplus_{i,j=1}^s M_{[i,j]}$ into
the direct sum of vector spaces $M_{[i,j]}=span \{e_{k,l} \mid
D_{i-1}<k\leq D_i,\; D_{j-1}<l\leq D_j \}$. Note that $M_{[i,j]}$
are submatrices supported on a single block of size $d_i \times
d_j$. This decomposition is natural in the sense that
$(F^{\alpha}H\otimes M_{r}(F))_g$ is the direct sum of the vector
spaces $U_h\otimes M_{[i,j]}$ such that $z_i^{-1}h z_j = g$.

For a fixed index $i \in \{1,\ldots,s\}$ and an element $g \in G$,
consider the equation $h z_j = z_i g$. It has a solution if and
only if $H z_i g$  has a representative in $(z_1 ,\ldots, z_s)$.
It follows that if $U_h \otimes B$ is homogeneous of degree $g$
and $H z_i g$ has no representative in $(z_1 ,\ldots, z_s)$, then
the $i$-th row of blocks in $B$ must be zero.

Consider the multilinear monomial
$$
x_{w_{1},1}x_{w_{2},2}\cdots x_{w_{n},n}
$$
where $x_{w_{i},i}$ is homogeneous of degree $w_{i} \in G$.
We will show there exist $w_{i} \in G, i=1,\ldots,n$, so that the monomial above is a $G$-graded identity.

To this end, note that such a monomial (being multilinear) is a $G$-graded identity if and
only if it is zero on graded assignments of the form
$x_{w_{i},i}=U_{h_{i}}\otimes A_{i}$ which span the algebra. In
particular the value of $x_{w_{1},1}x_{w_{2},2}\cdots x_{w_{n},n}$
under this assignment is $U_{h_{1}}\cdots U_{h_{n}}\otimes
A_{1}\cdots A_{n}$ which is zero if and only if $A=A_{1}\cdots
A_{n}=0$. It follows that if for any such homogeneous assignment,
the $i$-th row of blocks in the matrix $B_i = A_1 A_2 \cdots A_i$
is zero, then $A$ must be zero (since each of its blocks rows is
zero).

Following the argument above we choose $w_i \in G$ such that for
each $i$ the right coset $H z_i w_1 \cdots w_i $ (i.e. the right
coset of $H$ represented by $z_i$ times the homogeneous degree of
$U_{h_1}\cdots U_{h_i}\otimes B_i$) has no representative in $(z_1
,\ldots, z_s)$. Now, by assumption, there is some $z\in G$ such
that $Hz$ has no representative in $(z_1 ,\ldots, z_s)$. Thus,
choosing $w_i= (z_i w_1 \cdots w_{i-1})^{-1} z $ we obtain the
required result.

$(3)\rightarrow (2):$ Suppose that all right $H$-cosets are
represented in the tuple $(g_{1}, \ldots, g_{r})$. To show that
the grading is strong, it is enough to show that any basis element
$U_h\otimes e_{i,j}$ can be written as a product in
$A_{w_1}A_{w_2}$ where $w_1\cdot w_2=g_i ^{-1} h g_j$. Indeed,
since each right coset has a representative in the tuple $(g_1,\ldots,g_{r})\in G^{(r)}$, we can find $k$ such
that $g_{k} \in H g_i w_1=Hg_j w_2^{-1}$. Letting $h_1=g_{i} w_1
g_k^{-1}$ and $h_2 = g_k w_2 g_j ^{-1}$, we get that $a =
U_{h_1}\otimes e_{i,k},\; b = U_{h_2} \otimes e_{k,j}$  are in
$A_{w_1},A_{w_2}$ respectively and $a\cdot b= \alpha(h_1,h_2) U_h
\otimes e_{i,j}$. The lemma is now proved.
\end{proof}

Our next step is to pass from $G$-simple algebras to the group
algebra $FG$.

Let $V=\bigoplus V_g$ be a $G$-graded $F$-vector space. Then the
algebra of endomorphisms $End_{F}(V)$ has a natural $G$-grading
where an endomorphism $\psi\in End(V)$ has homogeneous degree $g$
if $\psi(V_h)\subseteq V_{gh}$ for every $h\in G$. In particular,
this grading on $End(FG)$ is isomorphic to the elementary grading
by a tuple $(g_1,\ldots,g_n)$ where each element of $G$ appears
exactly once. It is clear that the left regular action of $G$ on
$FG$ induces a natural $G$-graded embedding of $FG$ in
$End(FG)\cong M_{|G|}(F)$.

This statement can be generalized as follows.

\begin{lemma}
Let $G$ be a finite group, $H$ a subgroup and $
\{w_1,\ldots,w_k\}$ a complete set of representatives for the
right cosets of $H$ in $G$. Then the group algebra $FG$ can be
embedded in $FH\otimes M_k(F)$ where the tuple of the elementary
grading is $(w_1,\ldots,w_k)$.
\end{lemma}

\begin{proof}

For any $g\in G$ and any $H$-right coset representative $w\in
\{w_1,\ldots,w_k\}$, there are $h\in H$ and $w'\in
\{w_1,\ldots,w_k\}$ such that $wg=hw'$. We denote these elements
by $h:=h_{w,g}$ and $w':=w^g$. From associativity of $G$, we get
that
$$h_{w,g_1 g_2} = h_{w,g_1} h_{w^{g_1},g_2},\; \; w^{g_1 g_2}=(w^{g_1})^{g_2}.$$

Define a map $\psi:FG\to FH\otimes M_k(F)$ by
$$\psi (U_g) = \sum_{i=1}^{k} V_{h_{w_i,g}}\otimes E_{i,j(i)}$$
where $\{U_g\}_g$ and $\{V_h\}_h$ are the corresponding bases of
$FG$ and $FH$, $E_{i,j(i)}$ is the $(i,j)$ elementary matrix and
$j(i)$ is determined by the equation $w_{j(i)}= w_i^g$. It is easy
to show (left to the reader) that $\psi$ is a homomorphism.
Furthermore, by definition of the $G$-grading on $FH\otimes
M_k(F)$ (see Theorem \ref{thm:simple_graded}), we have that the
homogeneous degree of $V_{h_{w_i,g}}\otimes E_{i,j(i)}$ is
$w_i^{-1}h_{w_i,g}w_i^g=w_i^{-1}w_ig=g$, and hence $\psi$ is a
$G$-graded map. Finally, since $FG$ is $G$-simple and $\psi\neq
0$, it follows that $\psi$ is an embedding.

\end{proof}

Returning to our proof, we have a $G$-simple algebra $F^\alpha
H\otimes M_k(F)$, nondegenerately $G$-graded. Recall that this
means that any right coset of $H$ in $G$ appears at least once in
the tuple corresponding to the elementary grading $
\{w_1,\ldots,w_k\}$. If $\alpha=1$, then by the previous lemma the
group algebra $FG$ embeds in $FH\otimes M_k(F)$ and hence
$\Id(FG)\supseteq \Id(FH\otimes M_k(F))$. This proves the
reduction to $FG$ in that case.

In general (i.e. $\alpha$ not necessarily trivial), $FG$ may not
be $G$-graded embedded in $F^\alpha H\otimes M_k(F)$. We might
hope however, that even if such an embedding is not possible,
still $\exp(FG)\leq \exp(F^\alpha H\otimes M_k(F))$. It turns out
that this is also false as Example \ref{binary thetraeder} shows.

The next lemma shows how to get rid of the 2-cocycle $\alpha$.

\begin{lemma}
Let $A=F^\alpha H \otimes M_k(F)$ be a nondegenerate $G$-simple graded algebra.
Let $\rho:F^\alpha H\to M_d(F)$ be a nonzero $($ungraded$)$ representation
and denote by $B=M_d(F)$ the trivially $G$-graded algebra $($and therefore trivially $H$-graded$)$.
Then $FH$ can be embedded in $F^\alpha H\otimes B$ and
$FG$ can be embedded in $A\otimes B$ as $H$ and $G$-graded
algebras respectively.
\end{lemma}

\begin{proof}
Define the map $\psi:FH\to F^\alpha H\otimes M_d(F)$ by
$\psi(U_h)=V_h\otimes \rho(V_h^{-1})^t$, where $\{U_h\}_h$ and $\{V_h\}_h$ are
the corresponding bases of $FH$ and $F^\alpha H$. This is
easily checked to be an $H$-graded homomorphism, and it is an embedding
since $FH$ is $H$-simple. This proves the first claim of the lemma.

The second claim follows from the last lemma using the graded
embeddings
$$FG\hookrightarrow FH\otimes M_k(F)\hookrightarrow F^\alpha H\otimes M_d(F) \otimes M_k(F)\cong A\otimes B.$$

 \end{proof}

\begin{corollary} \label{reduction from G-simple to group algebra}
Let $A=F^\alpha H \otimes M_k(F)$ be a nondegenerate $G$-simple graded algebra.
Then $\exp(FG)\leq \exp(A)^2$.
\end{corollary}

\begin{proof}
Recall that the exponent of $F^\alpha H$ is $d^{2}$, where $d$ is the dimension of the
its largest irreducible representation. It follows that $FG$ can be
embedded in $A\otimes M_d(F)$ where $d^2=\exp(F^\alpha H)\leq \exp(A)$
and therefore
$$\exp(FG)\leq \exp(A\otimes M_d(F))= \exp(A) \exp(M_d(F))\leq\exp(A)^2.$$
\end{proof}

\begin{corollary} \label{bounding the exponent and PI degree finite groups}

Let $W$ be an associative PI $F$-algebra nondegenerately $G$-graded. Then the following hold.
$$\exp(FG) \leq \exp(W)^{2}$$
and
$$d(FG)\leq 2(d(W)-1)^{2}.$$ $($$d(W)$ denotes the PI degree$)$

\end{corollary}

\begin{proof}
The first inequality follows from Propositions \ref{reduction from arbitrary to finite dimensional algebras},
\ref{reduction from finite dimensional to G-simple} and Corollary \ref{reduction from G-simple to group algebra}.
For the proof of the second inequality recall that in general $\exp(W)\leq (d(W)-1)^2$
(see Theorem 4.2.4 in \cite{GiamZai}). Now, since $FG$ is semisimple, it follows by Amitsur-Levitzky theorem that $\frac{1}{2}d(FG)=\sqrt{\exp(FG)}$ and so, we conclude that

$$d(FG)=2\sqrt{\exp(FG)}\leq 2\exp(W)\leq 2(d(W)-1)^2.$$
\end{proof}

The last step in our analysis concerns with group algebras. Here
we refer to the following result of D. Gluck (see \cite{Glu}) in which he bounds
the minimal index of an abelian subgroup $U$ in $G$ in terms of
the maximal character degree of $G$. We emphasize that the proof
uses the classification of finite simple groups.

\begin{theorem}$($D. Gluck$)$
There exists a constant $m$ with the following property. For any
finite group $G$ there exists an abelian subgroup $U$ of $G$ such
that $[G:U] \leq b(G)^{m}$, where $b(G)$ is the largest
irreducible character degree of $G$.
\end{theorem}

We can now complete the proof of the main theorem for finite groups.

We note that by Giambruno and Zaicev's result  $\exp(FG)=b(G)^{2}$ and
hence, any finite group has an abelian subgroup $U$ with $[G:U]
\leq b(G)^{m} = \exp(FG)^{m/2}$. Combining with our results above,
we see that if a PI-algebra $W$ admits a nondegenerate $G$-grading
where $G$ is a finite group, then $\exp(FG)\leq \exp(W)^2$, hence
then there is an abelian subgroup $U$
with $[G:U] \leq \exp(W)^{m}$. In particular,
taking $K=m$ where $m$ is determined by the theorem above, will
do.

\end{section}

\begin{section}  {\label{sec:infinite} Proof of main theorem-Infinite groups}

In this section we prove the main theorem for arbitrary groups.
Let us sketch briefly the structure of our proof. In the preceding
section we proved the main theorem for arbitrary finite groups.
Our first step in this section is to prove the main theorem for
groups which are finitely generated and residually finite. Next,
we pass to finitely generated groups (not necessarily residually
finite) by the following argument. Any group $G$ which grades
nondegenerately a PI algebra is permutable and hence being
finitely generated, it is abelian by finite (see \cite{CLMR} or
Remark \ref{f.g. or non f.g. rewrite-permute}) and hence
residually finite. Finally we show how to pass from finitely
generated groups to arbitrary groups. We emphasize that the
constant $K$ (which appears in the main theorem) remains unchanged
when passing from finite groups to arbitrary group.

\begin{proposition}
Suppose the main theorem holds for arbitrary finite groups with
the constant $K$, that is, for any finite group $G$ and any PI
algebra $W$ which is nondegenerately $G$-graded, there exists an
abelian subgroup $U \subseteq G$ with $[G:U] \leq \exp(W)^{K}$.
Then the main theorem holds for finitely generated residually
finite groups with the same constant $K$.
\end{proposition}

\begin{proof}
Since $G$ is finitely generated, by Hall's theorem \cite{Ha} there
are finitely many subgroups of index $\leq\exp(A)^{K}$. Denoting
these groups by $U_{1},\ldots,U_{n}$, we wish to show that one of
them is abelian. Suppose the contrary holds. Hence we can find
$g_{i},h_{i}\in U_{i}$ such that $e\neq\left[g_{i},h_{i}\right]$
for any $i=1,\ldots,n$ and we let $N$ be a normal subgroup of
finite index which
doesn't contain any of the $\left[g_{i},h_{i}\right]$. \\
Define an induced $G/N$ grading on $A$ by setting
$A_{gN}=\bigoplus_{h\in N} A_{gh}$. Clearly, the induced
${G}/{N}$-grading on $A$ is nondegenerate, thus by the main
theorem there is some $U\leq G$ (containing $N$) such that
$\left[G:U\right]\leq\exp(A)^{K}$ and ${U}/{N}$ is abelian. By the
construction, $U=U_{i}$ for some $i$, and we get that
$\left[g_{i},h_{i}\right]\in N$ - a contradiction.
\end{proof}

The next step is to remove the condition of residually finiteness.

\begin{proposition}
Suppose the main theorem holds for finitely generated residually
finite groups with the constant $K$. Then the main theorem holds
for arbitrary finitely generated groups with the same constant
$K$.
\end{proposition}

\begin{proof}
As mentioned above this is obtained using permutability.

Let $G$ be a finitely generated group and suppose it grades
nondegenerately a PI algebra $A$. Let us show that $G$ must be
permutable. To this end let $f=\sum c_\sigma x_{\sigma(1),1}\cdots
x_{\sigma(n),n}$ be a nonzero ordinary identity of $A$ and assume
that $c_{id}=1$. Fix a tuple $g_1,\ldots,g_n\in G$ and consider
the graded identity
$$\tilde{f}=f(x_{g_1,1},\ldots,x_{g_n,n})=\sum _{h\in G} f_h(x_{g_1,1},\ldots,x_{g_n,n})$$
where ${f}_h$ is the $h$ homogenous part of $\tilde{f}$. Since
$\tilde{f}$ is a graded identity, its homogenous parts are also
graded identities. Letting $g=g_1 \cdots g_n$, the polynomial
${f}_g$ contains the monomial $x_{g_1,1}\cdots x_{g_n,n}$ (with
coefficient 1). Since the grading is nondegenerate, ${f}_g$ is not
a monomial and therefore has another monomial with nonzero
coefficient corresponding to some permutation $\sigma\neq id$,
hence $g_1\cdots g_n=g_{\sigma(1)}\cdots g_{\sigma(n)}$. This can
be done for any tuple of length $n$, so it follows that $G$ is
$n$-permutable.
\end{proof}

The main theorem now follows from the following proposition.

\begin{proposition} \label{pro:to_arbitrary_groups}
Let $G$ be any group and $d$ be a positive integer. Suppose that
any finitely generated subgroup $H$ of $G$ contains an abelian
subgroup $U_{H}$ with $[H:U_{H}] \leq d$. Then there exists an
abelian subgroup $U$ of $G$ with $[G:U] \leq d$.
\end{proposition}

\begin{remark}
The proposition above generalizes a statement which appears in
\cite{IsPass} but the proof is basically the same (see Lemma 3.5
and the proof of Theorem II). We believe the result of the
proposition is well known but we were unable to find an
appropriate reference in the literature.

\end{remark}

\begin{proof}

Let $A\leq F\leq G$. We say that $(F,A)$ is a pair if $F$ is f.g., $A$ is abelian and $[F:A]\leq d$. We write $(F,A) \leq (F_1,A_1)$ if $F \cap A_1 = A$.
Note in particular that $ [F:A] \leq [F_1:A_1]$.\\
A pair $(F,A)$ is called good if whenever $F\leq F_1 \leq G$ with $F_1$ finitely generated, there is a pair $(F_1,A_1)$ with $(F,A) \leq (F_1,A_1)$.
Note that the assumption of the proposition says that $({e},{e})$ is a good pair.\\
\begin{enumerate}

\item

We claim that if $(F,A)$ is good pair and $F\leq H \leq G$ with $H$ finitely generated, we can find $B\leq H$ such that $(H,B)$ is a good pair and $(F,A)\leq (H,B)$.\\
Indeed, since $(F,A)$ is a good pair, there are pairs $(H,B_i)$ with $(F,A)\leq (H,B_i)$, and by Hall's theorem there exist only finitely many such pairs.
Suppose by negation that none of them are good pairs.
Thus we can find $H\leq F_i \leq G$ ($F_i$-f.g.) such that there are
no abelian subgroups $A_i$ with $(H,B_i)\leq (F_i,A_i)$. The group $K=\left\langle F_1, ..., F_n \right\rangle$ is f.g. so there is
some abelian subgroup $A_K\leq K$ of index $\leq d$ such that $(F,A)\leq (K,A_K)$. Clearly, there is some $i$ such that $A_K\cap H=B_i$, but then $(H,B_i)\leq (F_i,F_i\cap A_K)$ - contradiction.

\item

Let $(F,A)$ be a good pair with $s=[F:A]$ maximal. Note that if
$(F,A)\leq (H,B)$ are good pairs, then we must have $[F:A]=[H:B]$.
Claim: for any such $B$ we have $[G:C_G(B)] \leq d$. Let us show
that if $g_1,\ldots,g_s$ represent the left cosets of $A$ in $F$
then they also represent the left cosets of $C_G(B)$ in $G$. Fix
an element $g\in G$. Then, by (1) above, $\langle H,g \rangle$ has
an abelian subgroup $C$ such that $(H,B)\leq (\langle H,g
\rangle,C)$ are good pairs. It follows that $g_1,\ldots,g_s$
represent also the left cosets of $C$ in $\langle H, g \rangle$
and hence $g \in g_i C$ for some $i$. Since $B\leq C$ are abelian
groups we get that $g_i C \subseteq g_i C_G(B)$ and the claim
follows.

\item

Assume now that $(F,A)$ is a good pair with $[F:A]=s$ and $[G:C_G(A)]$ maximal. Define
$$J= \langle B \mid (H,B)\geq (F,A) \mbox{ is a good pair} \rangle. $$ We claim that $J$ is abelian and $[G:J]\leq d$.\\
Let $(H_i,B_i)\geq (F,A)$, $i=1,2$, be good pairs, and let $b_i
\in B_i$. Since $A\leq B_1$, we have that $C_G(B_1)\leq C_G(A)$,
but from the maximality of $[G:C_G(A)]$, it follows that there is
an equality. Similarly, we have that $C_G(B_2) = C_G(A)$ and since
$B_2$ is abelian we get that $b_2\in B_2\subseteq
C_G(B_2)=C_G(B_1)$, so that $b_1,b_2$ commute.
This proves $J$ is abelian.\\
Suppose now that $[G:J]>d$, and let $g_0,\ldots,g_d$ different
coset representatives of $J$ in $G$. The group $F_1 = \langle F,
g_1,\ldots,g_d \rangle$ is finitely generated and so we can find
$A_1 \leq F_1$ such that $(F_1,A_1)$ is a good pair larger than
$(F,A)$, and in particular $[F_1:A_1]\leq d$. But this means that
there are some $0\leq i<j\leq d$ with $g_i ^{-1} g_j \in A_1
\subseteq J$ which is a contradiction. Thus, $[G:J] \leq d$ and we
are done.

\end{enumerate}
\end{proof}

As mentioned in the introduction, once a group has an abelian subgroup of finite index (say $d$), then it
also has a characteristic abelian subgroup of (finite) index bounded by a function of $d$. For completeness of the article we provide a simple proof here (shown to us by Uri Bader).

\begin{lemma} \label{characteristic subgroup}
There is a function $f:\mathbb{N}\to\mathbb{N}$ such that if a group
$G$ contains an abelian subgroup $A$ of index at most $n$, then
$G$ contains a characteristic abelian subgroup of index $\leq f(n)$.\end{lemma}
\begin{proof}
Let $N$ be the characteristic subgroup of $G$ generated by $A$ (the group generated by all images of $A$ under all automorphisms of $G$). Let $Z=Z(N)$ (the center of $N$). We claim $[N:Z]$ (and hence $[G:Z]$) is bounded by a function of $n$. Indeed, there are $n$ images of $A$ which already generate $N$ and $Z$ contains their intersection. This proves the lemma.

\end{proof}

In the preceding section, we proved that for a finite group and a nondegenerate
$G$-graded algebra $A$, we have $\exp(FG)\leq \exp(A)^2$ and $d(FG)\leq 2 (d(A)-1)^2$ .
The rest of this section is dedicated to generalize these results for infinite groups.

\begin{lemma}

Let $G$ be a finitely generated group such that
$FG$ is PI. Then there exists a finite index normal
subgroup $N$ of $G$ such that $Id(FG)=Id(FG/N)$.
In particular $Id(FG)=Id(M_k(F))$ for some integer $k$.

\end{lemma}

\begin{proof}
We know that $G$ is abelian by finite. Furthermore, since it is finitely generated it is residually finite. If $N$ is any finite index normal subgroup of $G$, then $Id(FG)\subseteq Id(FG/N)$ and
hence $I:=\bigcap_{[G:N]<\infty} Id(FG/N)\supseteq Id(FG)$. On the other
hand, the algebra $FG/N$ is semisimple, so that $Id(FG/N)=Id(M_k(F))$
where $k^2=\exp(FG/N)\leq \exp(FG)$ and in particular the set $\{\exp(FG/N)\}_{N}$ is bounded. It follows that $I=Id(M_k(F))$ for some $k$
and there is some finite index normal subgroup $N$ with $I=Id(FG/N)$.
The lemma will follow if we can show that $Id(FG)=I$.

Let $f(x_1,\ldots,x_m)\in I$ be any multilinear polynomial. If $f$
is not an identity of $FG$, we can find some $g_1,\ldots,g_m\in G$
such that $f(U_{g_1},\ldots,U_{g_m})\neq 0$ ($U_{g}$ represents
$g$ in $FG$). Let $h_1,\ldots,h_k\in G$ and $a_1,\ldots,a_k\in
F^\times$ such that $f(U_{g_1},\ldots,U_{g_m})=\sum a_i U_{h_i}$.
Since $G$ is residually finite, there is some finite index normal
subgroup $N$ not containing $h_i^{-1}h_j$ for any $i\neq j$.
Reducing the equation above modulo $N$, the elements $U_{h_i}$
remain linearly independent by the choice of $N$, so in particular
$f$ is not an identity of $FG/N$. We obtain that $I \subseteq \Id(FG)$ and the result follows.
\end{proof}

\begin{lemma}
Let $G$ be any group such that $FG$ is PI. Then there exists
some finitely generated subgroup $H$ of $G$ such that
$Id(FG)=Id(FH)$. Consequently, $Id(FG)=Id(M_k(F))$ for
some integer $k$.
\end{lemma}

\begin{proof}
By the preceding lemma, for each finitely generated subgroup $H$
of $G$ we have $Id(FH)=Id(M_k(F))$ for some integer $k$ which is
uniformly bounded over the finitely generated subgroups (by
$\exp(FG)^{1/2})$. Since any multilinear nonidentity of $FG$ is
already a nonidentity of $FH$ for some finitely generated subgroup
$H$ of $G$, we have that $Id(FG)=\bigcap_{H\leq G} Id(FH)$, $H$ is f.g.,
and the lemma follows.

\end{proof}

We can now generalize to arbitrary groups the result of the
previous section.

\begin{theorem}

Let $G$ be any group and $A$ be a nondegenerate $G$-graded algebra.
Then $\exp(FG)\leq \exp(A)^2$ and $d(FG)\leq 2 (d(A)-1)^2$.

\end{theorem}
\begin{proof}
Since $G$ grades nondegenerately the algebra $A$, it is abelian by
finite and therefore the group algebra $FG$ is PI. It follows from
the previous two lemmas that there is some finitely generated
subgroup $H$ and a finite index normal subgroup $N$ in $H$ such
that $Id(FG)=Id(FH/N)$, hence it is enough to bound the exponent
and PI-degree of $FH/N$.

If $A_H$ is the subalgebra of $A$ supported on the $H$ homogeneous
components of $A$, we have $\Id(A_{H}) \supseteq \Id(A)$,
$\exp(A_{H})\leq \exp(A)$ and $d(A_{H}) \leq d(A)$. Moreover,
since $A$ is nondegenerately $G$-graded, $A_{H}$ is
nondegenerately $H$-graded and hence $A_{H}$ is also
$H/N$-nondegenerately graded where $N$ is any normal subgroup of
$H$ (by the induced grading).

By Corollary \ref{bounding the exponent and PI degree finite groups} we have
$$\exp(FH/N)\leq \exp(A_H)^2 \leq \exp(A)^2$$

and

$$d(FH/N) \leq 2(d(A_{H})-1)^{2} \leq 2(d(A_{H})-1)^{2}$$

and the result follows.

\end{proof}

\end{section}
\begin{section} {\label{sec:examples} Some examples}

Let $G$ be a finitely generated group and suppose it grades
nondegenerately a PI algebra $A$. We know that $G$ is $n$
permutable for some $n \in \mathbb{N}$. While $G$ must be abelian
by finite, the minimal index of an abelian subgroup is not bounded
by a function of the permutablity index. Indeed, if there was such
a function $f(n)$, then given an arbitrary $n$-permutable group
$H$, its finitely generated subgroup would be $n$-permutable as
well. By the assumption, each such subgroup has an abelian
subgroup of index $\leq f(n)$, and hence, by Proposition
\ref{pro:to_arbitrary_groups}, the group $G$ would contain an abelian subgroup of index bounded by $f(n)$. This is known to be false. In fact $G$ need not have an abelian subgroup of finite index (see \cite{CLMR}).\\
Let us give a concrete example, i.e. a family of (finite) $n$-permutable groups $\{G_{k}\}_{k\in \mathbb{N}}$, with $d_{k}=min \{[G_{k}:U_{k}] \mid U_{k} $ abelian subgroup$\}$ and $\lim d_{k}= \infty$.

\begin{example}

Let $G=C_p^{2n}$ for some $n$ and let $\alpha \in
Z^2(G,\mathbb{C}^{*})$ be a \textit{nontrivial} two cocycle. It is
well known that up to a coboundary $\alpha$ takes values which are
roots of unity, and for $G$ above, the values must be $p$-roots of
unity. Thus, we may consider $\alpha$ as a cocycle in $Z^2(G,C_p)$
which corresponds to a central extension
$$1\to C_{p}\to  H\to  G\to 1.$$

Since the group $G$ is abelian and the cocycle $\alpha$ is nontrivial we have that $[H,H]=Z(H)\cong C_{p}$ and hence the group $H$ is $p+1$ permutable (see \cite{CLMR}, (3.3)).\\
Let $B=\mathbb{C}^\alpha G$ be the corresponding twisted group
algebra with basis $\{U_g\}_{g\in G}$. If $A\leq H$ is an abelian
group of minimal index, we have that $[H,H]\leq A$, and thus
we have $\tilde{A}=A/[H,H] \leq G$.
Clearly, the group $A$ is abelian if and only if $[U_{g_1},U_{g_2}]=1$ (the multiplicative commutator) for any $g_1,g_2 \in \tilde{A}$.\\
For $g,h\in G$, set $\mu(g,h)=\frac{\alpha(g,h)}{\alpha(h,g)}$,
namely the scalar satisfying $U_g U_h = \mu(g,h) U_h U_g$. It is
easily seen that $\mu:G\times G\to \mathbb{C}^{*}$ is a
bicharacter, i.e. $\mu(g_1 g_2,h)=\mu(g_1,h)\mu(g_2,h)$ and
$\mu(g,h_1 h_2)=\mu(g,h_1)\mu(g,h_2)$.
With this notation we have that $\mu(g_1,g_2)=1$ for any $g_1,g_2 \in \tilde{A}$.\\
Identifying $C_p$ with the additive group of the field $F_p$ with $p$ elements, we see that $\mu$ is a bilinear map. In particular, if $g \in G$, then $dim\{h \in G\mid \mu(h,g)=1\}\geq n-1$.
If  $dim_{F_p}(\tilde{A})>\frac{1}{2} \dim_{F_p} (G)$, or equivalently $[G:\tilde{A}]<p^n$, then there is some $e\neq u\in \tilde{A}$ such that $\mu(u,g)=1$ for all $g\in G$ (by dimension counting).
Thus, if $\mu$ is nondegenerate, i.e. for any $e\neq h\in G$ there is some $g\in G$ such that $\mu(h,g)\neq 1$, then $[H:A]=[G:\tilde{A}]\geq p^n$.\\
Note that to say that $\mu$ is nondegenerate is equivalent to
saying that $U_g$ is in the center of the twisted group algebra if
and only if $g=e$, which in turn is equivalent to the
twisted group algebras $\mathbb{C}^{\alpha}C_p ^{2n}$ being
isomorphic to a matrix algebra $M_{p^{n}}(\mathbb{C})$.

Fix a prime $p$ and let $\sigma, \tau$ be generators for $C_p
\times C_p$. Let $B$ be the twisted group algebra
$B=\bigoplus_{0\leq i,j \leq p-1} \mathbb{C}U_{\sigma^i \tau^j}$
where the multiplication is defined by
$$ U_{\sigma^i \tau^j} = {U_\sigma}^i{U_\tau}^j, \; U_\sigma U_\tau = \zeta U_\tau U_\sigma $$
and $\zeta$ is a primitive $p$-root of unity. It is well known
that $B\cong M_p(\mathbb{C})$, and hence $\bigotimes_1^n B$ is on
one hand isomorphic to a twisted group algebra with the group
$C_p^{2n}$ and on the other hand isomorphic to
$M_{p^{n}}(\mathbb{C})$. This completes the construction of the required family of groups. We remark here that the
function $\mu$ defined above plays a central role in the theory of
twisted group algebras and their polynomial identities (see
\cite{Al-David}).
\end{example}

\begin{remark}
Let $\alpha_n \in Z^2(C_p^{2n},\left\langle\zeta \right\rangle)$
be the nondegenerate 2-cocycle as constructed in the previous
example, and let $H_n$ be the central extensions defined by such
cocycle. The last example shows that the group algebra
$\mathbb{C}H_n$ has an irreducible representation of degree $p^n$.
On the other hand, Kaplansky's theorem \cite{Kap} states that if a
group has an abelian subgroup of index $m$, then all of its
irreducible representations are finite with degree at most $m$.
This provides another proof that the minimal index of an abelian
subgroup of $H_n$ tends to infinity.
\end{remark}

Next we provide some examples/counter examples to
statements that are related to the main theorem.

\begin{example}
Let F be an algebraically closed field of characteristic zero. For any finite abelian group
$G$, the group algebra $FG$ is isomorphic to a product of $|G|$ copies of $F$.
In particular, we get that $\exp(FG)=1$. Hence, we cannot hope to get an inequality
of the form $|G| \leq \exp(A)^K $ for any constant $K$. \\
More generally, given an $H$-graded algebra $A$ with a nondegenerate grading,
the algebra $B=FG\otimes A$ has a natural $G\times H$ grading which is also
nondegenerate. In addition we have that $\exp(B)=\exp(A)$. While the grading group
is of course larger, the index of the largest abelian group remains the same.

\end{example}

\begin{example}

Suppose we omit the requirement that $\Id_{G}(A)$ has no $G$-graded
monomials and only assume that $\Id_{G}(A)$ has $G$-graded monomials
of high degrees (as a function of $\dim(A)$ or the cardinality of
$G$). In other words we drop the assumption that $A$ is
nondegenerately $G$-graded and we only assume that the $G$-grading
on $A$ is nondegenerately bounded. We show that the
consequence of the main theorem does not hold in general.\\
Consider the algebras $A_{m}$ of upper triangular matrices
$m\times m$ where the diagonal matrices consist only of scalar
matrices. Note that by Giambruno and Zaicev's theorem (see \cite{GZ_affine}) we have $\exp(A)=1$. Let $G$ be a group
of order $n$ and assume that $m=n^{2}+1$. Let
$s'=(g_{1},\ldots,g_{n})\in G^{n}$ be a tuple such that each
element of $G$ appears in $s'$ exactly once and let $s\in
G^{\left(n^{2}+1\right)}$ be $n$ copies of $s'$ with additional
$g_{1}$ at the end. Consider the algebra $A_{m}$ with the
elementary grading corresponding to the tuple $s$. We claim that
$A_{m}$ has no graded multilinear monomial identities of degree
$\leq n$.\\
Fix $1\leq i\leq n^{2}+1-n$ and $h\in G$. We first note that by
the definition of the grading we have that $e_{i,j}$ is
homogeneous of degree $s_{i}^{-1}s_{j}$ for each $1\leq i<j\leq
n^{2}+1$. By the choice of the tuple $s$, the
elements $\left\{
s_{i}^{-1}s_{i+1},s_{i}^{-1}s_{i+2},\ldots,s_{i}^{-1}s_{i+n}\right\}
$ are all distinct, and therefore, for any $h \in G$ and $i\leq
n^{2}+1-n$ we can choose $j=j(i,h)$ such that $i<j\leq i+n$ and
$e_{i,j}\in A_{h}$.\\
Let $x_{h_{1},1}\cdots x_{h_{n},n}$ be any multilinear monomial, $h_{1},\ldots,h_{n}\in G$.\\
Set $i_{1}=1$. Given $i_{k}$, define $i_{k+1}$ to be
$j(i_{k},h_{k})$ so that $e_{i_{k},i_{k+1}}$ is homogeneous of
degree $h_{k}$ and $i_{k}<i_{k+1}\leq i_{k}+n$. It is now easy to
see by induction that $i_{k}\leq1+(k-1)n\leq1+n^{2}$ for all
$1\leq k\leq n$ so that $e_{i_{1},i_{2}}\cdots e_{i_{n},i_{n+1}}$
is well defined as an element of $A$ and it is a nonzero
evaluation
of $x_{h_{1},1}\cdots x_{h_{n},n}$.\\
 For a finite group $G$, denote by $\gamma(G)$ the smallest index
of an abelian subgroup in $G$. Let $G_{n}$ be any sequence of
groups where $\gamma(G_{n})$ goes to infinity with $n$. By the
above construction, the algebras $B_{n}=A_{|G_{n}|^{2}+1}$ have
$G_{n}$ gradings such that
\begin{itemize}
\item $\dim(B_{n})$ and $\gamma(G_{n})$ tend to infinity with $n$.
\item $B_{n}$ has no multilinear monomial identities of degrees smaller
then $|G_{n}|$.
\item $\exp(B_{n})=1$.
\end{itemize}
\end{example}

\begin{example}

Suppose we have a sequence of algebras $A_{n}$ with
$d_{n}=\exp(A_{n})$ monotonically increasing (i.e. to infinity).
Can we necessarily find groups $G_{n}$ and nondegenerate
$G_{n}$-gradings such that the index of any abelian subgroups
$U_{n}$ of $G_{n}$ tends to infinity? \\
The answer is negative as the algebras of upper triangular matrices show.
More precisely, let $UT_n(F)$ be the algebra of $n\times n$ upper
triangular matrices, which have exponent $\exp(UT_n(F))=n$.
By a theorem of Valenti and Zaicev \cite{VaZa},
every $G$-grading on $UT_n(F)$ is isomorphic to an elementary grading.
Unless the grading is trivial, the grading cannot be nondegenerate
since $UT_n(F)_g$ contains only upper triangular matrices with zero on the
diagonal for every $e\neq g\in G$, so $x_{g,1} \cdots x_{g,n}$ is an identity.
We conclude that the only nondegenerate grading is with the
trivial group, so in particular there are no nondegenerate grading
such that the index of the largest abelian subgroup tends to infinity.

\end{example}
\end{section}

\end{document}